\documentclass[12pt,reqno]{amsart}

\author{Richard Gottesman}
\address{Department of Mathematics and Statistics\\ Queen's University\\ Kingston, ON K7L3N6 }
\email{richard.b.gottesman@gmail.com}

\title{The module of vector-valued modular forms is Cohen-Macaulay}
\usepackage{amsmath,amssymb,amsthm}
\usepackage{geometry}
\usepackage{amsmath,amssymb,amsthm,url,pdfsync}
\usepackage{geometry}
\usepackage{thmtools}
\usepackage{thm-restate}
\usepackage{hyperref}
\usepackage[capitalise]{cleveref}

\newcommand{\nin}{\noindent}
\newcommand{\Z}{\textbf{Z}}
\newcommand{\Q}{\textbf{Q}}
\newcommand{\C}{\textbf{C}}

\newcommand{\Ind}{\textrm{Ind}}
\allowdisplaybreaks

\declaretheorem[name=Theorem, numberwithin=section]{thm}

\newtheorem{cor}[thm]{Corollary}
\newtheorem{lemma}[thm]{Lemma}
\newtheorem{prop}[thm]{Proposition}

\newtheorem{defi}[thm]{Definition}

\makeatletter
\def\moverlay{\mathpalette\mov@rlay}
\def\mov@rlay#1#2{\leavevmode\vtop{%
   \baselineskip\z@skip \lineskiplimit-\maxdimen
   \ialign{\hfil$\m@th#1##$\hfil\cr#2\crcr}}}
\newcommand{\charfusion}[3][\mathord]{
    #1{\ifx#1\mathop\vphantom{#2}\fi
        \mathpalette\mov@rlay{#2\cr#3}
      }
    \ifx#1\mathop\expandafter\displaylimits\fi}
\makeatother
\begin{document}

\begin{abstract} Let $H$ denote a finite index subgroup of the modular group $\Gamma$ and let $\rho$ denote a finite-dimensional complex representation of $H.$ Let $M(\rho)$ denote the collection of holomorphic vector-valued modular forms for $\rho$ and let $M(H)$ denote the collection of modular forms on $H$. Then $M(\rho)$ is a $\Z$-graded $M(H)$-module. It has 
been proven that $M(\rho)$ may not be projective as a $M(H)$-module. We prove that $M(\rho)$ is Cohen-Macaulay as a $M(H)$-module. We also explain how to apply this result to prove that if $M(H)$ is a polynomial ring then $M(\rho)$ is a free $M(H)$-module of rank $\textrm{dim } \rho.$ \end{abstract}
\maketitle


\section{Introduction} 
\nin Let $H$ denote a finite index subgroup of the modular group $\Gamma := \textrm{SL}_{2}(\Z)$ and let $\rho$ denote a finite-dimensional complex representation of $H.$ Let $k \in \Z$ and let $\mathfrak{H}$ denote the complex upper half plane.
If $F: \mathfrak{H} \rightarrow \C^{t}$ is a holomorphic function and if $\gamma =  \left[ {\begin{array}{cc}
   a & b \\
   c & d \\
\end{array} } \right] \in \Gamma $ then we define $F|_{k} \gamma$ by setting
 $F|_{k}  \gamma (\tau) :=  (c \tau + d)^{-k} F\left(\frac{a \tau + b}{c \tau + d}\right).$
 \begin{defi} A \textbf{vector-valued modular form }$F$ of weight $k$ with respect to $\rho$ is a holomorphic function $F: \mathfrak{H} \rightarrow \C^{\textrm{dim } \rho}$ which is also holomorphic at all of the cusps of \newline 
$H \backslash (\mathfrak{H} \bigcup \mathbb{P}^{1}(\Q))$ and such that for all  $\gamma \in H$, 
\begin{align}
F|_{k} \gamma =  \rho(\gamma)F. 
\end{align}
 \end{defi}

\nin The statement that $F$ is holomorphic at all of the cusps of $H \backslash (\mathfrak{H} \bigcup \mathbb{P}^{1}(\Q))$  is equivalent to the statement that for each $\gamma \in \Gamma$, each of the component functions of $F|_{k} \gamma$ has a holomorphic $q$-expansion. The notion of a holomorphic $q$-expansion is a bit more intricate in the vector-valued setting. A precise treatment of the notion of a holomorphic $q$-expansion can be found for $\Gamma$ in \cite{survey} and for an arbitrary subgroup in \cite{gottesman}.\\

\nin The collection of all vector-valued modular forms of weight $k$ for the representation $\rho$ form a finite-dimensional complex vector-space, which we denote by $M_{k}(\rho)$. We let $M_{t}(H)$ denote the collection of all modular forms of weight $t$ on $H.$  We define $M(\rho) := \bigoplus_{k \in \Z} M_{k}(\rho)$ and $M(H) := \bigoplus_{k \in \Z} M_{k}(H).$ If $F \in M_{k}(\rho)$ and if $m \in M_{t}(H)$ then $mF \in M_{k+t}(\rho).$ In this way, we view $M(\rho)$ as a 
$\Z$-graded $M(H)$-module. If $\rho$ is a representation of $\Gamma$ then the module structure of $M(\rho)$ is especially pleasing. 

\begin{thm} \label{MarksMason} Let $\rho$ denote a representation of $\Gamma$. Then $M(\rho)$ is a free 
$M(\Gamma)$-module of rank equal to the dimension of $\rho$. 
\end{thm}  

\nin Theorem \ref{MarksMason} has been used to study the arithmetic of vector-valued modular forms for representations of $\Gamma$ in \cite{marksfourier}, \cite{franc2013fourier}, \cite{Mason2012}. There are multiple proofs of Theorem \ref{MarksMason} and each one offers its own perspective and insights. Theorem \ref{MarksMason} was proven by Marks and Mason in \cite{marksmason}, by Terry Gannon in \cite{gannon2014theory}, and by Candelori and Franc in \cite{francfreemodule}. \\ 

\nin Mason has shown that if $H$ is equal to $\Gamma^{2}$, the unique subgroup of $\Gamma$ of index two, then $M(\rho)$ need not be a free module over $M(\Gamma^{2}).$ A proof of this fact together with the result that $M(\rho)$ need not even be projective over $M(\Gamma^{2})$ appears in Section $6$ of \cite{structure}. 
In view of this negative result, it is natural to ask if one may prove a positive result about the structure of $M(\rho)$ as a $M(H)$-module. We prove the following:  

\begin{thm} \label{main} $M(\rho)$ is Cohen-Macaulay as a $M(H)$-module. 
\end{thm}

\nin We shall also explain how to apply Theorem \ref{main} to prove the following theorem. 
\begin{thm} \label{free} Let $H$ denote a finite index subgroup of $\Gamma$ such that there exist modular forms $X, Y \in M(H)$ for which $X$ and $Y$ are algebraically independent 
and $M(H) = \C[X,Y].$ Then $M(\rho)$ is a free $M(H)$-module of rank $\textrm{dim } \rho.$
\end{thm}

\nin We remark that Theorem \ref{free} may also be obtained by applying the work of Candelori and Franc in \cite{structure}. A complete list of the finitely many subgroups $H$ which satisfy the hypothesis of Theorem \ref{free} is given in \cite{polynomial}. Two such subgroups are $\Gamma$ and $\Gamma_{0}(2).$ The author employs Theorem \ref{free} to study the arithmetic of vector-valued modular forms on $\Gamma_{0}(2)$ in \cite{gottesman}. \\

\nin In \cite{structure}, Candelori and Franc study the commutative algebra properties of vector-valued modular forms in a geometric context. If $H$ is a genus zero Fuchsian group of the first kind, with finite covolume and with finitely many cusps, then they define a collection of \textit{geometrically weighted} vector-valued modular forms, $\textrm{GM}(\rho)$, which contains $M(\rho)$, and a collection of geometrically weighted modular forms $S(H)$, which contains $M(H).$ They prove that
 $\textrm{GM}(\rho)$ is Cohen-Macaulay as a $S(H)$-module. The ideas in \cite{structure} involve the classification of vector bundles over orbifold curves of genus zero and are quite interesting. We emphasize that in our paper, Theorem \ref{main} applies to the $M(H)$-module $M(\rho)$ and holds for \textit{all} finite index subgroups of $\Gamma.$ \\

\nin We refer the reader to Benson \cite{benson} for the relevant definitions and results from commutative algebra which we shall use in this paper.

\section{Acknowledgements} 
\nin The inspiration for this paper was the proof of the free module theorem in \cite{marksmason}. 
It is my pleasure to thank Cam Franc, Geoff Mason, Matt Satriano, and Noriko Yui for their interest and encouragement.  
 
\section{Proofs} 
\nin The following lemma will be used in the proof of Theorem \ref{main} and Theorem \ref{free}. 

\begin{lemma} \label{induction} Let $\textrm{Ind}_{H}^{\Gamma} \rho$ denote the induction of the representation $\rho$ from $H$ to $\Gamma.$ Then $M(\rho)$ and $M(\textrm{Ind}_{H}^{\Gamma} \rho)$ are isomorphic as $\Z$-graded $M(\Gamma)$-modules. 
\end{lemma}
\begin{proof} Let $\{g_{i}: 1 \leq i \leq [\Gamma:H] \}$ denote a complete set of left coset representatives of $H$ in $\Gamma$ where $g_1$ denotes the identity matrix. Let $k \in \Z$ and let $F \in M_{k}(\rho).$
We define $\Phi(F) := [F|_{k} g_{1}^{-1}, F|_{k} g_{2}^{-1},..., F|_{k} g_{[\Gamma: H]}^{-1}]^{T},$ where the superscript $T$ denotes the transpose. We claim that $\Phi(F) \in M_{k}(\textrm{Ind}_{H}^{\Gamma} \rho).$ We first note that since $F \in M_{k}(\rho)$, we must have that for all $g \in H$, the function $F|_{k} g$ is holomorphic in $\mathfrak{H}$ and it has a holomorphic $q$-expansion. Thus $\Phi(F)$ is holomorphic in $\mathfrak{H}$ and $\Phi(F)$ is holomorphic at the cusp  
$\Gamma \backslash (\mathfrak{H} \bigcup \mathbb{P}^{1}(\Q))$ since each of its component functions $F|_{k} g_{i}^{-1}$ has a holomorphic $q$-expansion. To prove that $\Phi(F) \in M_{k}(\textrm{Ind}_{H}^{\Gamma} \rho)$, it now suffices to show that for all  $g \in \Gamma$,  
\begin{align}
\label{inductionone}
\Phi(F)|_{k} g = \textrm{Ind}_{H}^{\Gamma} \rho(g) \Phi(F).
\end{align}

\nin Let $n = [\Gamma: H].$ Let $\rho^{\bullet}$ denote the function on $\Gamma$ which is defined by the conditions that $\rho^{\bullet}|_{H} = \rho$ and $\rho^{\bullet}(g) = 0$ if $g \not \in H.$ With respect to our choice of left coset representatives for $H$ in $\Gamma$, the $i$-th row and $j$-th column block of the matrix $\textrm{Ind}_{H}^{\Gamma} \rho(g)$ is equal to $\rho^{\bullet}(g_{i}^{-1}gg_{j}).$ Thus equation \ref{inductionone}
is equivalent to the assertion that for each integer $i$ with $1 \leq i \leq n$, 
\begin{align}
\label{inductiontwo}
(F|_{k}g_i^{-1})|_{k} g =\sum_{t=1}^{n} \rho^{\bullet}(g_i^{-1} g g_{t}) F|_{k} g_{t}^{-1}.
\end{align}
\nin Fix an index $i$ with $1 \leq i \leq n$  and fix $g \in \Gamma.$ Then there exists a unique index $j$ for which $g_i^{-1} g g_j \in H.$
\nin We then have that $F|_{k} g_i^{-1}g g_j = \rho(g_i^{-1} g g_j)F.$
Therefore 
\begin{align}
 (F|_{k}g_i^{-1})|_{k} g 
& = F|_{k} g_i^{-1}g & \\
& = (F|_{k} g_i^{-1}gg_j)|_{k} g_{j}^{-1}\\
 & = (\rho(g_i^{-1} g g_j)F)|_{k} g_{j}^{-1} \\
& = \rho(g_i^{-1} g g_j) F|_{k} g_j^{-1} \\
& =\sum_{t=1}^{n} \rho^{\bullet}(g_i^{-1} g g_{t}) F|_{k} g_{t}^{-1}. \end{align}
\nin We have thus proven that equation \ref{inductiontwo} holds and conclude that $\Phi(F) \in M_{k}(\textrm{Ind}_{H}^{\Gamma}(\rho)).$
\nin For each integer $k$, we have defined the map $\Phi: M_{k}(\rho) \rightarrow M_{k}(\textrm{Ind}_{H}^{\Gamma}(\rho))$ and we extend it by linearity to a map $\Phi: M(\rho) \rightarrow M(\textrm{Ind}_{H}^{\Gamma} \rho).$ \\

\nin We now check that $\Phi$ is a map of $\Z$-graded $M(\Gamma)$-modules. 
Let $m \in M_{t}(\Gamma).$ We recall that $F \in M_{k}(\rho).$ 
Then \begin{align} m \Phi(F) & = m  [F|_{k} g_{1}^{-1}, F|_{k} g_{2}^{-1},..., F|_{k} g_{n}^{-1}]^{T} \\
& = [mF|_{k+t} g_{1}^{-1}, mF|_{k+t} g_{2}^{-1},..., mF|_{k +t} g_{n}^{-1}]^{T} \\
& = \Phi(mF).
\end{align}
\nin Finally, we check that $\Phi$ is a bijection. Let $X \in M_{k}(\textrm{Ind}_{H}^{\Gamma} \rho).$ The codomain of $X$ is $\C^{\textrm{dim } (\textrm{Ind}_{H}^{\Gamma} \rho)}$ and 
$\textrm{dim} (\textrm{Ind}_{H}^{\Gamma} \rho) = n \textrm{dim } \rho.$
Let $X_1,...,X_n: \mathfrak{H} \rightarrow \C^{\textrm{dim } \rho}$ denote the holomorphic functions for which $X = [X_1,...,X_n]^{T}.$
We define the map $\pi$ by setting $\pi(X) = X_{1}$. We will show that $X_1 \in M_{k}(\rho).$
Let $g \in H$. There exists a unique index $j$ for which $g_1^{-1}gg_{j} \in H$. As $g_1$ is the identity matrix and $g \in H$, we must have that $g_j = g_1.$
The fact that $X \in M_{k}(\Ind_{H}^{\Gamma} \rho)$ implies that 
\begin{align}
X_1|_{k} g = \sum_{t=1}^{n} \rho^{\bullet}(g_1^{-1} g g_{t}) X_t = \rho^{\bullet}(g_1^{-1}gg_1)X_1 = \rho(g)X_{1}. \end{align}

\nin As $X = [X_1,...,X_{n}]^{T}$ is a holomorphic vector-valued modular form, $X_1$ is holomorphic in $\mathfrak{H}$ and $X_1$ has a holomorphic $q$-expansion. Thus for each $g \in \Gamma$, $X_1|_{k} g$ has a holomorphic $q$-expansion since 
$X_1|_{k} g = \rho(g) X_1$ and $X_1$ has a holomorphic $q$-series expansion. We have proven that $X_1$ is holomorphic at all of the cusps of $H \backslash (\mathfrak{H} \bigcup \mathbb{P}^{1}(\Q))$ and conclude that $\pi(X) = X_{1} \in M_{k}(\rho).$ We also see that $\pi \circ \Phi(F) = \Phi(F)_{1} = F|_{k} g_1^{-1} = F$ since $g_1$ is the identity matrix.  As $\pi \circ \Phi$ is equal to the identity map, $\pi$ must be surjective.  All that remains is to prove that $\pi$ is injective. \\

\nin Let $X \in M_{k}(\textrm{Ind}_{H}^{\Gamma} \rho)$ such that $\pi(X) = 0$. We write $X = [ X_1, X_2,...,X_n ]^{T}$ where each $X_i$ is a holomorphic function from $\mathfrak{H}$ to $\C^{\textrm{dim } \rho}.$ 
We claim that $X = 0$. Suppose not. Then there exists some index $i$ with $X_i \neq 0$. 
Let $g \in g_i H g_1^{-1}$. Then $g_i^{-1}gg_j \in H$ if and only if $g_{1} = g_{j}.$
Thus $X_{i}|_{k} g = \sum_{t=1}^{n} \rho^{\bullet}(g_i^{-1}gg_{t})X_{t} = \rho(g_{i}^{-1}gg_{1})X_{1}.$
As $\pi(X) = X_1 = 0$, we have that $X_{i}|_{k} g = 0$. 
Hence $X_i = (X_i|_{k}g)|_{k} g^{-1} = 0$, a contradiction. Thus $X = 0$ and $\pi$ is therefore injective. 
We have proven that $\pi$ and hence $\Phi$ is a bijection and the lemma now follows. 
 \end{proof}

\nin Let $q = e^{2 \pi i \tau}$, $\sigma_{3}(n) = \sum_{d \mid n} d^{3}$, and let $S_{k}(\Gamma)$ denote the space of weight $k$ cusp forms on $\Gamma.$
\nin We recall the Eisenstein series $E_4 = 1 + 240 \sum_{n=1}^{\infty} \sigma_{3}(n)q^{n} \in M_{4}(\Gamma)$ and the cusp form $\Delta = q(1-q^n)^{24} \in S_{12}(\Gamma).$  
\begin{lemma} The sequence $\Delta, E_4$ is a regular sequence for the $M(H)$-module $M(\rho).$
\end{lemma}

\begin{proof} It is clear that $\Delta$ is a non-zero-divisor for $M(\rho)$ since $\Delta$ has no zeros in $\mathfrak{H}.$ To prove that $\Delta$ is regular for $M(\rho)$, it suffices to show that 
$M(\rho) \neq \Delta M(\rho).$  Suppose that $M(\rho) = \Delta M(\rho)$. Let $X$ denote a non-zero element in $M(\rho)$ of minimal weight $w$. Then $X = \Delta V$ for some $V \in M_{w - 12}(\rho)$. 
The weight of $V$ is less than weight of $X$. This is a contradiction and therefore $M(\rho) \neq \Delta M(\rho).$ We conclude that $\Delta$ is regular for $M(\rho)$. \\

\nin We will show that $E_4$ is regular for $M(\rho)/ \Delta M(\rho)$. We have previously shown that  \newline 
\nin $M(\rho)/ \Delta M(\rho) \neq 0.$ We now argue that $E_4$ is a non-zero-divisor for the module $M(\rho)/ \Delta M(\rho)$. Suppose that $Y \in M(\rho)$ and $E_4(Y + \Delta M(\rho)) = \Delta M(\rho)$. 
Then $E_4 Y \in \Delta M(\rho)$.  We write $E_4 Y = \Delta Z$ for some $Z \in M(\rho).$  We wish to show that $Y \in \Delta M(\rho)$ and it suffices to prove this when $Y$ is a vector-valued modular form. Let $k$ denote the weight of $Y$. Let $y_i$ denote the $i$-th component function of $Y$, let $z_i$ denote the $i$-th component function of $Z$, and let $\gamma \in \Gamma$. Therefore $E_4 y_i = \Delta z_i$ and 
$E_{4} (y_{i}|_{k} \gamma) = \Delta (z_{i}|_{k-8} \gamma).$ As $\Delta = q + O(q^2)$ and $E_4 = 1 + O(q)$, all the powers of $q$ in $y_{i}|_{k} \gamma$ occur to at least the first power. We have thus shown that $\Delta^{-1}(y_i|_{k} \gamma)$ contains no negative powers of $q$ and is therefore holomorphic at the cusp $\gamma \cdot \infty$. Hence $\frac{Y}{\Delta}$ is holomorphic at all of the cusps of $H \backslash (\mathfrak{H} \bigcup \mathbb{P}^{1}(\Q))$  . As $\Delta$ does not vanish in $\mathfrak{H}$, we have that $\frac{Y}{\Delta}$ is holomorphic in $\mathfrak{H}$. Hence $\frac{Y}{\Delta} \in M(\rho)$ and thus $Y + \Delta M(\rho) = \Delta M(\rho).$  We have proven that $E_4$ is a non-zero-divisor for the module $M(\rho)/ \Delta M(\rho).$ \\

\nin Finally, we must show that $E_{4}(M(\rho)/\Delta M(\rho)) \neq M(\rho)/ \Delta M(\rho).$ 
We recall that $X$ denotes a nonzero element in $M(\rho)$ of minimal weight $w.$ If $M(\rho)/\Delta M(\rho) = E_{4} (M(\rho)/ \Delta M(\rho))$ then there exists some $F \in M(\rho)$ 
such that $X + \Delta M(\rho) = E_4 F + \Delta M(\rho).$ Let $G \in M(\rho)$ such that $X =  E_4 F + \Delta G.$ We may write $F$ and $G$ uniquely as a sum of their homogeneous components. Let $F_{w - 4}$ and 
$G_{w - 12}$ denote the weight $w -4$ and the weight $w - 12$ homogeneous components of $F$ and $G.$ Then $X = E_4 F_{w -4} + \Delta G_{w - 12}.$ We must have that $F_{w - 4} \neq 0$ or $G_{w -12} \neq 0$ since 
$X \neq 0.$ Thus $F_{w-4}$ or $G_{w -12}$ is a nonzero element of $M(\rho)$ whose weight is less than the weight of $X.$ This is a contradiction and we conclude that $M(\rho)/\Delta M(\rho) \neq E_{4} (M(\rho)/ \Delta M(\rho)).$ We have shown that  $E_4$ is regular for $M(\rho)/\Delta M(\rho)$ and our proof is complete. \end{proof}

\begin{lemma} \label{krull} The Krull dimension of the $M(H)$-module $M(\rho)$ is equal to two. 
\end{lemma}
\begin{proof} We recall that the Krull dimension of the $M(H)$-module $M(\rho)$ is defined to be the Krull dimension of the ring $M(H)/\textrm{Ann}_{M(H)}(M(\rho))$. As the zeros of a nonzero holomorphic function are isolated, $\textrm{Ann}_{M(H)} M(\rho) = 0$. Therefore the Krull dimension of $M(\rho)$ is equal to the Krull dimension of $M(H)$. It suffices to prove that the Krull dimension of $M(H)$ is equal to the Krull dimension of $M(\Gamma) = \C[E_4, E_6]$, which is equal to two. To do so, 
we use the fact (see Corollary $1.4.5$ in Benson \cite{benson}) that if $A \subset B$ are commutative rings and if $B$ is an integral extension of $A$ for which $B$ is finitely generated as an $A$-algebra then the Krull dimensions of $A$ and $B$ are equal. It therefore suffices to show that $M(H)$ is an integral extension of $M(\Gamma)$ and that $M(H)$ is finitely generated as a $M(\Gamma)$-algebra. \\

\nin  Let $\{\gamma_{i}: 1 \leq i \leq [\Gamma:H] \}$ denote a complete set of right coset representatives of $H$ in 
$\Gamma$ where $\gamma_1$ denotes the identity matrix.
If $f \in M_{k}(H)$  then $f|_{k}\gamma_1 = f$ and thus $f$ is a root of the monic polynomial $P(z) := \prod_{i=1}^{[\Gamma:H]} (z - f|_{k} \gamma_{i})  \in M(H)[z].$
If $\gamma \in \Gamma$ then let $P(z)|_{k} \gamma$ denote the polynomial obtained by replacing each monomial $cz^{t}$ of $P$ with the monomial $(c|_{k} \gamma) z^{t}.$
The fact that $\{\gamma_i \gamma: 1 \leq i \leq [\Gamma: H]\}$ is a complete set of right coset representatives of $H$ in $\Gamma$ together with the fact that $f \in M_{k}(H)$ 
implies that 
\begin{align}
\label{integral}
 P(z)|_{k} \gamma = \prod_{i=1}^{[\Gamma:H]} (z - f|_{k} \gamma_i \gamma) = \prod_{i=1}^{[\Gamma:H]} (z - f|_{k} \gamma_{i}) = P(z).
\end{align}
Thus $P(z) \in M(\Gamma)[z].$  Hence $M(H)$ is an integral extension of $M(\Gamma).$ \\

\nin Let $\alpha: H \rightarrow \C^{\times}$ denote the trivial representation of $H.$ 
Theorem \ref{MarksMason} implies that $M(\textrm{Ind}_{H}^{\Gamma} \alpha)$ is a free $M(\Gamma)$-module whose rank equals $\textrm{dim}(\textrm{Ind}_{H}^{\Gamma} \alpha) = [\Gamma: H].$
Lemma \ref{induction} tells us that $M(\alpha)$ and $M(\textrm{Ind}_{H}^{\Gamma} \alpha)$ are isomorphic as $M(\Gamma)$-modules.  Hence $M(\alpha) = M(H)$ is a free $M(\Gamma)$-module of rank $[\Gamma:H].$ In particular, $M(H)$ is finitely generated as a $M(\Gamma)$-algebra. We conclude that the Krull dimensions of $M(H)$ and $M(\Gamma)$ are equal and the lemma now follows. 
 \end{proof}

\nin We now proceed with the proof of Theorem \ref{main}.

\begin{proof} \nin We have shown that the Krull dimension of the $M(H)$-module $M(\rho)$ is equal to two and that $M(\rho)$ has a regular sequence of length two. 
Therefore the depth of $M(\rho)$ is at least two. Moreover, the depth is at most the Krull dimension (see page $50$ in \cite{benson}), which is equal to two. 
Hence the depth and the Krull dimension of $M(\rho)$ are both equal to two.

\end{proof}

\nin We shall use the following result from commutative algebra to prove Theorem \ref{free}. This result is stated and proven in Benson's book \cite{benson}. \\

\begin{thm} \label{comm} [{Theorem $4.3.5.$ in \cite{benson}}] Let $A$ denote a commutative Noetherian ring and let $M$ denote a finitely generated $A$-module. Assume that $A = \bigoplus_{j=0}^{\infty} A_{j}$ and $M = \bigoplus_{j= -\infty}^{\infty} M_{j}$ are graded, $A_0 = K$ is a field, and $A$ is finitely generated over $K$ by elements of positive degree. Then the following statements are equivalent: 
\begin{enumerate} 
\item[(i):] $M$ is Cohen-Macaulay. 
\item [(ii):] If $x_1,...x_n \in A$ are homogenous elements generating a polynomial subring \newline 
\nin $K[x_1,...,x_n] \subset A/\textrm{Ann}_{A}(M)$, over which $M$ is finitely generated, then
$M$ is a free $K[x_1,...,x_n]$-module. 
\end{enumerate} 
\end{thm}

\nin We now give the proof of Theorem \ref{free}.
\begin{proof} 
We first note that the hypotheses of Theorem \ref{comm} are satisfied if we take $A = M(H)$ and $M = M(\rho).$
Thus statements (i) and (ii) in Theorem \ref{comm} are equivalent if $A = M(H)$ and $M = M(\rho).$
We have proven that $M(\rho)$ is Cohen-Macaulay as a $M(H)$-module.  Thus statement (i) and hence statement (ii) in Theorem \ref{comm} must be true. In particular, if $X, Y \in M(H)$ which are algebraically independent then $M(\rho)$ is a free $\C [X,Y]$-module.  The hypothesis of Theorem \ref{free} asserts that there exist such modular forms $X$ and $Y$ for which $M(H) = \C[X,Y].$ Thus the hypothesis of Theorem \ref{free} implies that $M(\rho)$ is a free $M(H)$-module. \\

\nin  We now compute the rank $r$ of $M(\rho)$ as a $M(H)$-module. We have have shown in the proof of Lemma \ref{krull} that $M(H)$ is a free $M(\Gamma)$-module of rank $[\Gamma:H].$ Therefore $M(\rho)$ is a free $M(\Gamma)$-module of rank $[\Gamma:H]r.$ Theorem \ref{MarksMason} tells us that $M(\textrm{Ind}_{H}^{\Gamma} \rho)$ is a free $M(\Gamma)$-module whose rank equals $\textrm{dim}(\textrm{Ind}_{H}^{\Gamma} \rho) = [\Gamma: H] \textrm{dim } \rho.$ Lemma \ref{induction} states that $M(\rho)$ and $M(\textrm{Ind}_{H}^{\Gamma} \rho)$ are isomorphic as $M(\Gamma)$-modules.  Thus $M(\rho)$ is a free $M(\Gamma)$-module whose rank equals $[\Gamma: H] \textrm{dim } \rho.$ Hence $[\Gamma:H]r = [\Gamma: H] \textrm{dim } \rho$ and $r = \textrm{dim } \rho.$

\end{proof} 

\bibliographystyle{plain}
\bibliography{CohenM}

\begin{thebibliography}{10}

\bibitem{polynomial}
Eiichi Bannai, Masao Koike, Akihiro Munemasa, and Jiro Sekiguchi.
\newblock Some results on modular forms---subgroups of the modular group whose
  ring of modular forms is a polynomial ring.
\newblock In {\em Groups and combinatorics---in memory of {M}ichio {S}uzuki},
  volume~32 of {\em Adv. Stud. Pure Math.}, pages 245--254. Math. Soc. Japan,
  Tokyo, 2001.

\bibitem{benson}
David~J. Benson.
\newblock {\em Polynomial invariants of finite groups}, volume 190.
\newblock Cambridge University Press, 1993.

\bibitem{structure}
Luca Candelori and Cameron Franc.
\newblock Vector bundles and modular forms for {Fuchsian} groups of genus zero.
\newblock {\em arXiv preprint arXiv:1704.01684}, 2017.

\bibitem{francfreemodule}
Luca Candelori and Cameron Franc.
\newblock Vector-valued modular forms and the modular orbifold of elliptic
  curves.
\newblock {\em International Journal of Number Theory}, 13(01):39--63, 2017.

\bibitem{franc2013fourier}
Cameron Franc and Geoffrey Mason.
\newblock Fourier coefficients of vector-valued modular forms of dimension 2.
\newblock {\em Canadian Mathematical Bulletin}, 57(3):485--494, 2014.

\bibitem{survey}
Cameron Franc and Geoffrey Mason.
\newblock Hypergeometric series, modular linear differential equations and
  vector-valued modular forms.
\newblock {\em The Ramanujan Journal}, 41(1-3):233--267, 2016.

\bibitem{gannon2014theory}
Terry Gannon.
\newblock The theory of vector-valued modular forms for the modular group.
\newblock In {\em Conformal Field Theory, Automorphic Forms and Related
  Topics}, pages 247--286. Springer, 2014.

\bibitem{thesis}
Richard Gottesman.
\newblock {\em The algebra and arithmetic of vector-valued modular forms on
  {$\Gamma_{0}(2)$}}.
\newblock PhD thesis, University of California, Santa Cruz, 2018.

\bibitem{gottesman}
Richard Gottesman.
\newblock The arithmetic of vector-valued modular forms on {$\Gamma_{0}(2)$}.
\newblock {\em arXiv preprint arXiv:1811.04452}, 2018.

\bibitem{marksfourier}
Christopher Marks.
\newblock Fourier coefficients of three-dimensional vector-valued modular
  forms.
\newblock {\em Commun. Number Theory Phys.}, 9(2):387--412, 2015.

\bibitem{marksmason}
Christopher Marks and Geoffrey Mason.
\newblock Structure of the module of vector-valued modular forms.
\newblock {\em Journal of the London Mathematical Society}, 82(1):32--48, 2010.

\bibitem{Mason2012}
Geoffrey Mason.
\newblock On the {Fourier} coefficients of $2$-dimensional vector-valued
  modular forms.
\newblock {\em Proceedings of the American Mathematical Society},
  140(6):1921--1930, 2012.

\end{thebibliography}

\end{document}